\documentclass[12pt]{amsart}
\usepackage{graphicx}
\usepackage[headings]{fullpage}
\usepackage{bm,amssymb,epic,eepic,epsfig,amsbsy,amsmath,amscd}
\numberwithin{equation}{section}
                        \theoremstyle{plain}
\usepackage{mathrsfs}

\newcommand\no[1]{}

\newtheorem{theorem}{Theorem}[section]
\newtheorem{thm}{Theorem}
\newtheorem{lemma}[theorem]{Lemma}

\newtheorem{proposition}[theorem]{Proposition}

\theoremstyle{definition}

\newcommand{\la}{\langle}
\newcommand{\ra}{\rangle}

\def\BC{\mathbb C}

\def\BZ{\mathbb Z}
\def\BR{\mathbb R}

\def\la{\langle}
\def\ra{\rangle}

\def\be { \begin{equation} }
\def\ee { \end{equation} }

\begin{document}

\title[non LO surgeries on L-space twisted torus knots]
{Non left-orderable surgeries on L-space twisted torus knots}

\author{Anh T. Tran}

\begin{abstract}
We show that if $K$ is an L-space twisted torus knot $T^{l,m}_{p,pk \pm 1}$ with $p \ge 2$, $k \ge 1$, $m \ge 1$ and $1 \le l \le p-1$, then the fundamental group of the $3$-manifold obtained by $\frac{r}{s}$-surgery along $K$ is not left-orderable whenever $\frac{r}{s} \ge 2 g(K) -1$, where $g(K)$ is the genus of $K$.
\end{abstract}



\address{Department of Mathematical Sciences, The University of Texas at Dallas, Richardson, TX 75080, USA}
\email{att140830@utdallas.edu}

\maketitle

\section{Introduction}

Heegaard Floer homology is a package of 3-manifold invariants introduced by Ozsvath and Szabo \cite{OS}. Manifolds with minimal Heegaard Floer homology are called L-spaces.  More precisely, a rational homology 3-sphere $Y$ is an L-space if the hat version of its Heegaard Floer homology, denoted by $\widehat{HF}(Y)$, has rank equal to the order of $H_1(Y; \BZ)$. Since computing $\widehat{HF}$ is difficult, it is desirable to know if there are characterizations of L-spaces which do not refer to Heegaard Floer homology. The L-space conjecture of Boyer, Gordon and Watson \cite{BGW} proposes such a characterization. It states that an irreducible rational homology 3-sphere is an L-space if and only if its fundamental group is not left-orderable. Here a nontrivial group $G$ is left-orderable if it admits a total ordering $<$ such that $g<h$ implies $fg < fh$ for all elements $f,g,h$ in $G$. 

Many hyperbolic L-spaces can be obtained via Dehn surgery. A knot $K$ in $S^3$ is called an L-space knot if it admits a positive Dehn surgery yielding an L-space. For an L-space knot $K$, it was shown in  \cite{OS} that the $\frac{r}{s}$-surgery of $K$ is an L-space if and only if $\frac{r}{s}\ge 2g(K)-1$, where $g(K)$ is the genus of $K$. In view of the L-space conjecture, one would expect that the fundamental group of the $\frac{r}{s}$-surgery of an L-space knot $K$ is not left-orderable if and only if $\frac{r}{s} \ge 2g(K)-1$. When an L-space knot $K$ is the $(-2,3,2n+1)$-pretzel knot (with $n \ge 3$) or the $(n-2)$-twisted $(3,3m+2)$-torus knot (with $n, m \ge 1$), it was shown in \cite{Nie, Tr} that  the fundamental group of the $\frac{r}{s}$-surgery of $K$ is not left-orderable if $\frac{r}{s} \ge 2g(K)-1$. 

We denote by $T^{l,m}_{p,q}$ the twisted torus knot obtained from the $(p,q)$-torus knot by
twisting $l$ strands $m$ full times. Determining all L-space knots among twisted torus knots is an open problem. However, when $q \equiv \pm 1 \pmod{p} $, this problem was already solved. For $p \ge 2$, $k \ge 1$, $m \ge 1$ and $1 \le l \le p-1$, it was shown in \cite{Va} that the twisted torus knot $T^{l,m}_{p, pk \pm 1}$ is an L-space knot if and only if $l=p-1$ or $l \in \{p-2, 2\}$ and $m=1.$ Some results on the non left-orderability for surgeries of  L-space twisted torus knots $T^{p-1,m}_{p, pk \pm 1}$ and $T^{p-2,1}_{p, pk \pm 1}$ were obtained in \cite{Jun, Nakae, CW, IT2015, IT2018, C}. In this paper, we generalize the above results. More precisely, we will show the following. 
  
\begin{thm} \label{main}
Let $K$ be an L-space twisted torus knot $T^{l,m}_{p,pk \pm 1}$ with $p \ge 2$, $k \ge 1$, $m \ge 1$ and $1 \le l \le p-1$. Then the fundamental group of the $3$-manifold obtained by $\frac{r}{s}$-surgery along $K$ is not left-orderable if $\frac{r}{s} \ge 2 g(K) -1.$
\end{thm}

The paper is organized as follows. In Section \ref{prelim} we prove some preliminary lemmas on the knot groups of twisted torus knots $T^{l,m}_{p, pk \pm 1}$ and we compute the genera of the L-space knots among these knots. In Section \ref{pf} we use the results of the previous section to prove Theorem \ref{main}.

\section{Preliminary lemmas} \label{prelim}

For a knot $K \subset S^3$, the knot group of $K$ is defined to be the fundamental group of the complement of $K$ in $S^3$. We denote by $G^{l,m}_{p,q}$ the knot group of $T^{l,m}_{p, q}$. From now on we consider the twisted torus knots $T^{l,m}_{p,pk \pm 1}$ with $p \ge 2$, $k \ge 1$, $m \ge 1$ and $1 \le l \le p-1$.

\subsection{The case $T^{l,m}_{p,pk-1}$} \label{pre-1}
By \cite{C} $G^{l,m}_{p,pk-1}$ has a presentation with two generators $a,b$ and one relation
\begin{equation} \label{rel-ab}
a^{p-l+1} (z^{m} a)^{l-1} = b^{k(p-l+1)-1} (z^m b^k)^{l-1}
\end{equation}
where $z = b^{1 - k(p-l)} a^{p-l}.$ Moreover, $\mu = a^{-1}b^k$ is a meridian and the preferred longitude corresponding to $\mu$ is $\lambda = \mu^{- p(pk-1) - l^2 m}  \lambda'$ where $\lambda' = a^{p-l} (z^m a )^l.$ 

Let $x = \mu = a^{-1}b^k$ and $y=b^{1-k}a$. Then $b=yx$, $a = b^{k-1} y = (yx)^{k-1}y$ and $z= (yx)^{1 - k(p-l)} ( (yx)^{k-1}y)^{p-l}$. Moreover, $G^{l,m}_{p,pk-1}$ has a presentation with two generators $x,y$ and one relation
$$(yx)^{k-1} y ((yx)^{k-1} y)^{p-l} (z^{m}  (yx)^{k-1}y )^{l-1} = (yx)^k (yx)^{k(p-l)-1} (z^m (yx)^k)^{l-1}.$$
This relation is equivalent to 
\begin{equation} \label{rel-1}
((yx)^{k-1} y)^{p-l} (z^{m}  (yx)^{k-1}y )^{l-1} = x (yx)^{k(p-l)-1} (z^m (yx)^k)^{l-1}.
\end{equation}
By abelianizing $G^{l,m}_{p,pk-1}$, the homology class of $y$ is $p-1$ times that of $x$. 

Let $\text{Homeo}^+(\BR)$ be the group of order-preserving homeomorphisms of $\BR.$

\begin{lemma} \label{ineq-1}
Suppose $\rho: G^{l,m}_{p,pk-1} \to \text{Homeo}^+(\BR)$ is a homomorphism such that $\rho(x)\alpha > \alpha$ for all $\alpha \in \BR$. Then $\rho(y)\alpha > \alpha$ and 
\begin{eqnarray*}
\rho ( ( (yx)^{k-1} y )^{p-l} ) \alpha &>& \rho (  x (yx)^{k(p-l)-1}  ) \alpha, \\
\rho (  (yx)^{k-1} y^n ) \alpha &>& \rho (  x^n (yx)^{k-1} ) \alpha, \\
\rho (  y^n (xy)^{k-1} ) \alpha &>& \rho (  (xy)^{k-1}  x^n ) \alpha
\end{eqnarray*}
 for all $\alpha \in \BR$ and $n \ge 1$. 
 \end{lemma}

\begin{proof}
Since $\rho(x)\alpha > \alpha$ for all $\alpha \in \BR$ we have $\rho ( z^m (yx)^k ) \alpha  > \rho ( z^m (yx)^{k-1} y ) \alpha.$ The relation  \eqref{rel-1} then implies that $$\rho ( ( (yx)^{k-1} y )^{p-l} ) \alpha > \rho (  x (yx)^{k(p-l)-1}  ) \alpha$$
 for all $\alpha \in \BR$. Since $\rho (  (yx)^k ) \alpha  > \rho (  (yx)^{k-1} y ) \alpha$ we have
 $$\rho ( ( (yx)^{k-1} y )^{p-l} ) \alpha > \rho (  x (yx)^{k(p-l)-1}  ) \alpha>  \rho (  x (yx)^{k-1} ( (yx)^{k-1} y )^{p-l-1} ) \alpha.$$
 Hence $\rho ( (yx)^{k-1} y ) \alpha > \rho (  x (yx)^{k-1} ) \alpha $ for all $\alpha \in \BR$. By \cite[Lemma 2.4]{Tr}, this implies that 
\begin{eqnarray*}
\rho (  (yx)^{k-1} y^n ) \alpha &>& \rho (  x^n (yx)^{k-1} ) \alpha, \\
\rho (  y^n (xy)^{k-1} ) \alpha &>& \rho (  (xy)^{k-1}  x^n ) \alpha
\end{eqnarray*}
for all $\alpha \in \BR$ and $n \ge 1$. 

Finally, since $\rho ( (yx)^{k-1} y ) \alpha > \rho (  x (yx)^{k-1} ) \alpha >  \rho (   (yx)^{k-1} ) \alpha$ we have $\rho(y)\alpha > \alpha.$
\end{proof}

\begin{lemma} \label{genus-1}
If $T^{l,m}_{p,pk-1}$ is an L-space knot, then its genus is equal to $\frac{1}{2}l(l-1)m + \frac{1}{2} p(p-1)k-(p-1).$
\end{lemma}

\begin{proof}
For a Laurent polynomial $f(t) \in \BC[t^{\pm 1}]$, the degree of $f$ is defined to be the difference of the highest degree and the lowest degree of $f$. 

By \cite{OS}, the degree of the  Alexander polynomial of an L-space knot is twice the knot genus. Hence, to prove Lemma \ref{genus-1} it suffices to show that the degree of the Alexander polynomial of $T^{l,m}_{p,pk-1}$ is equal to $l(l-1)m +  p(p-1)k-(2p-2).$ We will compute the Alexander polynomial  via Fox's free calculus. 

Since $z = b^{1 - k(p-l)} a^{p-l}$, the relation \eqref{rel-ab} is equivalent to $za (z^m a)^{l-1} = b^k (z^m b^k)^{l-1}.$ Let $r_1 = za(z^m a)^{l-1}$ and $r_2 = ax(z^max)^{l-1}.$ Since $b^k = ax$ we have  $G^{l,m}_{p,pk-1} = \la x,y \mid r_1 = r_2 \ra$. The Alexander polynomial of $T^{l,m}_{p,pk-1}$ can be computed via the formula
\begin{equation} \label{Alex}
\frac{\Delta_{T^{l,m}_{p,pk-1}}(t)}{t-1}  =  \frac{\Phi(\frac{\partial r_1}{\partial x} - \frac{\partial r_2}{\partial x})}{\Phi(y)-1}
\end{equation}
where $\Phi: G^{l,m}_{p,pk-1} \to \BZ \cong \la t \ra$ is the abelianization. See, for example, \cite[Section 3]{Tr}. 

Since $\Phi(x)=t$ and $\Phi(y) = t^{p-1}$, we have $\Phi(a) = t^{kp-1}$, $\Phi(b) = t^p$ and $\Phi(z) = \Phi( b^{1 - k(p-l)} a^{p-l}) = t^l.$ We will use the following formulas $$\frac{\partial u v}{\partial x} = \frac{\partial u}{\partial x} + u \frac{\partial v}{\partial x}, \quad \frac{\partial u^{-1}}{\partial x} = -u^{-1} \frac{\partial u }{\partial x} \quad \text{and} \quad  \frac{\partial u^n}{\partial x} = ( 1 + u + \cdots + u^{n-1}) \frac{\partial u }{\partial x}$$
for $u,v \in G^{l,m}_{p,pk-1}$ and $n \ge 1$. Note that $\frac{\partial x }{\partial x} =1$ and $\frac{\partial y}{\partial x} =0$.

For $n \ge 0$ and $u \in G^{l,m}_{p,pk-1} $ we let $\delta_n(u) = \sum_{i=0}^n u^i.$ By direct calculations we have $\Phi(\frac{\partial a}{\partial x}) = \Phi(\delta_{k-2}(yx) y) = \frac{t^{p(k-1)}-1}{t^p-1} t^{p-1}$, $\Phi(\frac{\partial b}{\partial x}) = \Phi(y) = t^{p-1}$ and 
\begin{eqnarray*}
\Phi \Big( \frac{\partial z}{\partial x} \Big) &=& \Phi \Big( - b^{1-k(p-l)} \delta_{k(p-l)-2}(b) \frac{\partial b}{\partial x} + b^{1-k(p-l)} \delta_{p-l-1}(a) \frac{\partial a}{\partial x} \Big) \\
&=& t^{p-kp(p-l)} \Big( \frac{t^{(p-l)(kp-1)} -1}{t^{kp-1} -1} \cdot \frac{t^{p(k-1)} -1}{t^{p} -1}
-\frac{t^{kp(p-l)-p} -1}{t^{p} -1} \Big) t^{p-1}.
\end{eqnarray*}

Since
\begin{eqnarray*}
\frac{\partial r_1}{\partial x} &=& \frac{\partial z}{\partial x} + z \frac{\partial a}{\partial x} + za \delta_{l-2}(z^m a) \Big(\delta_{m-1}(z) \frac{\partial z}{\partial x} + z^m \frac{ \partial a}{\partial x} \Big),\\
\frac{\partial r_2}{\partial x} &=& \frac{\partial a}{\partial x} + a + ax 
 \delta_{l-2}(z^max) \Big(\delta_{m-1}(z) \frac{\partial z}{\partial x} + z^m \frac{\partial a}{\partial x}  + z^m a \Big),
\end{eqnarray*}
we have
\begin{eqnarray*}
\frac{\partial r_1}{\partial x} - \frac{\partial r_2}{\partial x} 
&=& \Big( z + za \delta_{l-2}(z^m a) z^m - 1 - ax 
 \delta_{l-2}(z^max) z^m  \Big) \frac{\partial a}{\partial x} \\
&& + \, \Big( 1 + za \delta_{l-2}(z^m a)  \delta_{m-1}(z) - ax 
 \delta_{l-2}(z^max)  \delta_{m-1}(z) \Big)  \frac{\partial z}{\partial x} \\
 && - \, a - ax \delta_{l-2}(z^max)  z^m a.
\end{eqnarray*} 

By comparing degrees, the highest degree of the  Laurent polynomial $\Phi \big( \frac{\partial r_1}{\partial x} - \frac{\partial r_2}{\partial x}  \big)$ is $m l(l-1) + kp l - 1$, which is the highest degree of $\Phi(ax \delta_{l-2}(z^max)  z^m a)$. Moreover, the lowest degree is $-kp(p-l-1)+p-1$, which is the lowest degree of $\Phi \big( \frac{\partial z}{\partial x} \big)$. Hence, the degree of $\Phi \big( \frac{\partial r_1}{\partial x} - \frac{\partial r_2}{\partial x}  \big)$ is $l(l-1)m + p(p-1)k-p$. Equation \eqref{Alex} then implies that the degree of the Alexander polynomial of $T^{l,m}_{p,pk-1}$ is equal to 
$l(l-1)m + p(p-1)k-(2p-2).$
\end{proof}

\subsection{The case $T^{l,m}_{p,pk+1}$} \label{pre+1}
By \cite{C} $G^{l,m}_{p,pk+1}$ has a presentation with two generators $a,b$ and one relation
\begin{equation} \label{rel-ab+1}
a (z^{m} a)^{l-1} a^{p-l} = b^k (z^m b^k)^{l-1} b^{k(p-l)+1}
\end{equation}
where $z = b^{k(p-l)+1} a^{l-p}.$ Moreover, $\mu = b^{-k}a$ is a meridian and the preferred longitude corresponding to $\mu$ is $\lambda = \mu^{- p(pk+1) - l^2 m}  \lambda'$ where $\lambda'=(z^m a )^l a^{p-l}.$ 

Let $x = \mu = b^{-k}a$ and $y=a^{-1}b^{k+1}$. Then $b=xy$ and $a = b^{k} x = (xy)^{k}x.$ Moreover, $G^{l,m}_{p,pk+1}$ has a presentation with two generators $x,y$ and one relation
$$ (xy)^{k}x ( z^{m}  (xy)^{k}x )^{l-1} ((xy)^{k}x)^{p-l}  = (xy)^k ( z^m (xy)^k)^{l-1} (xy)^{k(p-l)+1}$$
where $z= (xy)^{k(p-l)+1} ( (xy)^{k}x)^{l-p}.$ The group relation is equivalent to 
\begin{equation} \label{rel+1}
(  x z^m (xy)^k )^{l-1} x( (xy)^k x )^{p-l}  = ( z^{m} (xy)^{k} )^{l-1} (xy)^{k(p-l)+1}.
\end{equation}

By abelianizing $G^{l,m}_{p,pk+1}$, the homology class of $y$ is $p-1$ times that of $x$. Similar to the previous case, we have the following lemmas.

\begin{lemma} \label{ineq+1}
Suppose $\rho: G^{l,m}_{p,pk+1} \to \text{Homeo}^+(\BR)$ is a homomorphism such that $\rho(x)\alpha > \alpha$ for all $\alpha \in \BR$. Then $\rho(y)\alpha > \alpha$ and
\begin{eqnarray*}
\rho ( (xy)^{k(p-l)+1} ) \alpha &>& \rho (  x ( (xy)^k x )^{p-l} ) \alpha, \\
\rho (  (yx)^{k} y^n ) \alpha &>& \rho (  x^n (yx)^{k} ) \alpha, \\
\rho (  y^n (xy)^{k} ) \alpha &>& \rho (  (xy)^{k}  x^n ) \alpha
\end{eqnarray*}
 for all $\alpha \in \BR$ and $n \ge 1$. 
 \end{lemma}

\begin{lemma} \label{genus+1}
If $T^{l,m}_{p,pk+1}$ is an L-space knot, then its genus is equal to $\frac{1}{2}l(l-1)m + \frac{1}{2} p(p-1)k.$
\end{lemma}

\subsection{Longtitude of L-space $T^{l,m}_{p,pk \pm 1}$} From now on we let $K$ denote an L-space twisted torus knot $T^{l,m}_{p, pk \pm 1}.$ Since $K$ is an L-space, by \cite{Va}, one of the following conditions holds

$\quad$ (1) $l=p-1$, 

$\quad$ (2) $l=p-2$ and $m=1$, 

$\quad$ (3) $l=2$ and $m=1$. 

\begin{lemma} \label{long}
If $K=T^{p-1,m}_{p, pk - 1}$ then $\lambda' =(yx)^{k-1} y^{m+1} x (yx)^{k-1} (y^m (yx)^k)^{p-2}.$

If $K=T^{p-1,m}_{p, pk + 1}$ then $\lambda' =   ((xy)^{k} x y^m )^{p-1} (xy)^k x.$

If $K=T^{p-2,1}_{p, pk - 1}$ then $\lambda' = x (yx)^{k-1} (y (yx)^{k-1}y )^{p-2} (yx)^{k-1}y.$

If $K=T^{p-2,1}_{p, pk + 1}$ then $\lambda' =  (xy)^{k} x (  y (xy)^{k} )^{p-2} (xy)^k x.$

If $K=T^{2,1}_{p, pk - 1}$ then
\begin{eqnarray*} \label{long}
\lambda' &=&( (yx)^{k-1} y )^{p-2}  [ (yx)^{1-(p-2)k} ( (yx)^{k-1} y )^{p-1} ]^{2}  \\
  &=& x ( (yx)^{k-1}y )^{p-1}x ( (yx)^{k-1} y )^{-(p-1)}  (yx)^k ( (yx)^{k-1}y )^{p-1} x.
\end{eqnarray*} 
  
If $K=T^{2,1}_{p, pk + 1}$ then 
\begin{eqnarray*}
\lambda' &=&  (xy)^{(p-2)k+1} ( (xy)^{k} x)^{-(p-3)}  (xy)^{(p-1)k+1} x \\
&=&  x (xy)^{(p-1)k+1}  x  (xy)^{-(p-1)k-1}  (xy)^{k} x (xy)^{(p-1)k+1}  x.
\end{eqnarray*} 
\end{lemma}

\begin{proof}
If $K=T^{p-1,m}_{p, pk - 1}$ then, by Section \ref{pre-1}, we have $z = y$ and $\lambda' = (yx)^{k-1} y  ( y^m (yx)^{k-1} y )^{p-1}.$ The relation  \eqref{rel-1} is $(yx)^{k-1} y (y^{m}  (yx)^{k-1}y )^{p-2} = x (yx)^{k-1} (y^m (yx)^k)^{p-2}.$ Hence 
\begin{eqnarray*}
\lambda' &=&  (yx)^{k-1} y  y^m (yx)^{k-1} y ( y^m (yx)^{k-1} y )^{p-2}  \\
&=& (yx)^{k-1} y^{m+1} x (yx)^{k-1} (y^m (yx)^k)^{p-2}.
\end{eqnarray*} 

If $K=T^{p-1,m}_{p, pk + 1}$ then, by Section \ref{pre+1}, we have $z = ((xy)^{k} x) y ((xy)^{k} x)^{-1}$ and $\lambda' =  ( (xy)^{k} x y^m )^{p-1} (xy)^k x.$

If $K=T^{p-2,1}_{p, pk - 1}$ then, by Section \ref{pre-1}, we have $z = (yx)^{1-2k} ( (yx)^{k-1}y )^{2} = (yx)^{-k} y (yx)^{k-1}y$ and $\lambda' = ((yx)^{k-1} y)^2 (z (yx)^{k-1} y)^{p-2}.$
The relation \eqref{rel-1} is $((yx)^{k-1} y)^{2} (z(yx)^{k-1}y )^{p-3} = x (yx)^{2k-1} (z (yx)^k)^{p-3}.$ Hence 
\begin{eqnarray*}
\lambda' &=& ( (yx)^{k-1} y )^{2} (z   (yx)^{k-1} y )^{p-3} z   (yx)^{k-1} y\\
  &=& x (yx)^{2k-1} (z  (yx)^k )^{p-3} z(yx)^{k-1}y \\
    &=& x (yx)^{k-1} ( (yx)^k z)^{p-2} (yx)^{k-1}y \\
  &=& x (yx)^{k-1} (y (yx)^{k-1}y)^{p-2} (yx)^{k-1}y.
\end{eqnarray*} 

If $K=T^{p-2,1}_{p, pk + 1}$ then, by Section \ref{pre+1}, we have $z = (xy)^{2k+1} ( (xy)^{k}x )^{-2}$ and 
\begin{eqnarray*}
\lambda' &=& [ (xy)^{2k+1} ( (xy)^{k}x )^{-1} ]^{p-2} ((xy)^k x)^2 \\
   &=& [ (xy)^{k}x y (xy)^k  ( (xy)^{k}x )^{-1} ]^{p-2} ((xy)^k x)^2\\
   &=& (xy)^{k}x (y (xy)^k)^{p-2}  ( (xy)^{k}x )^{-1} ((xy)^k x)^2 \\
   &=& (xy)^{k} x (  y (xy)^{k} )^{p-2} (xy)^k x.
\end{eqnarray*} 

If $K=T^{2,1}_{p, pk - 1}$ then, by Section \ref{pre-1}, we have $z = (yx)^{1-(p-2)k} ( (yx)^{k-1}y )^{p-2}$ and 
$$
\lambda' =( (yx)^{k-1} y )^{p-2}  [ (yx)^{1-(p-2)k} ( (yx)^{k-1} y )^{p-1} ]^{2}.
$$
The relation \eqref{rel-1} is $( (yx)^{k-1} y )^{p-2} (yx)^{1-(p-2)k}  ( (yx)^{k-1} y )^{p-1} =  x ( (yx)^{k-1}y )^{p-1} x.$
Hence
\begin{eqnarray*}
\lambda' 
 &=& ( (yx)^{k-1} y )^{p-2}  [  ( (yx)^{k-1} y)^{-(p-2)} x ( (yx)^{k-1}y )^{p-1} x ]^2 \\
 &=& x ( (yx)^{k-1}y )^{p-1} x  ( (yx)^{k-1} y)^{-(p-2)} x ( (yx)^{k-1}y )^{p-1} x  \\
 &=& x ( (yx)^{k-1}y )^{p-1}x ( (yx)^{k-1} y )^{-(p-1)}  (yx)^k ( (yx)^{k-1}y )^{p-1} x.
\end{eqnarray*} 

If $K=T^{2,1}_{p, pk + 1}$ then, by Section \ref{pre+1}, we have $z = (xy)^{(p-2)k+1} ( (xy)^{k}x )^{-(p-2)}$ and 
$$
\lambda' =  [ (xy)^{(p-2)k+1} ( (xy)^{k} x )^{-(p-3)} ]^{2} ( (xy)^{k} x )^{p-2}.
$$
The  relation \eqref{rel+1} is
$x (xy)^{(p-1)k+1}  x  =  (xy)^{(p-2)k+1} ( (xy)^{k}x )^{-(p-2)}  (xy)^{(p-1)k+1}.$
Hence
\begin{eqnarray*}
\lambda' 
&=& (xy)^{(p-2)k+1} ( (xy)^{k} x )^{-(p-3)} (xy)^{(p-2)k+1}  (xy)^{k} x \\
&=& (xy)^{(p-2)k+1} ( (xy)^{k} x )^{-(p-2)} (xy)^{k} x (xy)^{(p-1)k+1}  x \\
&=&  x (xy)^{(p-1)k+1}  x  (xy)^{-(p-1)k-1}  (xy)^{k} x (xy)^{(p-1)k+1}  x.
\end{eqnarray*}
This completes the proof of Lemma \ref{long}.
\end{proof}

\section{Non left-orderable surgeries} \label{pf}

In this section we prove Theorem \ref{main}. Recall that $K$ is an L-space twisted torus knot $T^{l,m}_{p, pk \pm 1}.$ Let $M_{\frac{r}{s}}$ be the 3-manifold obtained by $\frac{r}{s}$-surgery along $K$. Assume $\pi_1(M_{\frac{r}{s}})$ is left-orderable for some $\frac{r}{s} \ge 2g(K)-1$, where $s>0$. Then there exists a homomorphism $\rho: \pi_1(M_\frac{r}{s}) \to \text{Homeo}^+(\BR)$ such that there is no $\alpha \in \BR$ satisfying $\rho(h)(\alpha) = \alpha$ for all $h \in \pi_1(M)$, see e.g. \cite[Problem 2.25]{CR}. From now on we write $h \alpha$ for $\rho(h)(\alpha)$.

Since $x^r \lambda^s =1$ in $\pi_1(M)$ and $x\lambda = \lambda x$, there exists an element $v \in \pi_1(M)$ such that $x= v^s$ and $\lambda = v^{-r}$, see e.g. \cite[Lemma 3.1]{Nakae}. 

\begin{proposition}
We have $v \alpha \not=\alpha$ for all $\alpha \in \BR$.
\end{proposition}

\begin{proof}
Assume $v \alpha=\alpha$ for some $\alpha \in \BR$. Then $x\alpha=v^s \alpha=\alpha$ and $\lambda \alpha=v^{-r} \alpha=\alpha.$ Since $\lambda' = x^{p(pk \pm 1) + l^2 m}  \lambda$, we have
$\lambda' \alpha = \alpha.$ If $y\alpha=\alpha$ then $h\alpha=\alpha$ for all $h \in \pi_1(M)$, a contradiction. Without loss of generality, we assume that $y\alpha > \alpha$. 

If $K= T^{p-1,m}_{p,pk \pm 1}$ or $K= T^{p-2,1}_{p,pk \pm 1}$ then, by Lemma \ref{long}, all the exponents of $y$ in $\lambda'$ are positive. Since $x \alpha = \alpha$ and $y\alpha > \alpha$, we conclude that $\lambda' \alpha > \alpha.$ This contradicts $\lambda' \alpha = \alpha.$

If $K= T^{2,1}_{p,pk - 1}$ then, by Lemma \ref{long}, we have 
\begin{eqnarray*}
\lambda' &=&( (yx)^{k-1} y )^{p-2}  [ (yx)^{1-(p-2)k} ( (yx)^{k-1} y )^{p-1} ]^{2}  \\
  &=& x ( (yx)^{k-1}y )^{p-1}x ( (yx)^{k-1} y )^{-(p-1)}  (yx)^k ( (yx)^{k-1}y )^{p-1} x.
\end{eqnarray*} 
Let $u=( (yx)^{k-1}y )^{p-1}$. Since $x \alpha = \alpha$ and $y\alpha > \alpha$ we have $u \alpha > \alpha$. Hence 
\begin{eqnarray*}
&&\alpha = \lambda' \alpha = x u xu^{-1} (yx)^k u x \alpha \\
\Rightarrow && \alpha = ux u^{-1} (yx)^k u \alpha \\
\Rightarrow &&  (yx)^k u \alpha = u x^{-1} u^{-1} \alpha < u x^{-1}  \alpha = u \alpha \\
\Rightarrow &&  (yx) u \alpha < u \alpha \\
\Rightarrow &&  (yx)^{-1} u \alpha > u \alpha.
\end{eqnarray*}
This implies that  $(yx)^{1-(p-2)k} ( (yx)^{k-1} y )^{p-1} \alpha = (yx)^{1-(p-2)k} u \alpha > u \alpha > \alpha.$ Hence 
$$\lambda' \alpha = ( (yx)^{k-1} y )^{p-2}  [ (yx)^{1-(p-2)k} ( (yx)^{k-1} y )^{p-1} ]^{2} \alpha > \alpha.$$
This contradicts $\lambda' \alpha =\alpha.$

If $K= T^{2,1}_{p,pk + 1}$ then, by Lemma \ref{long}, we have 
\begin{eqnarray*} 
\lambda' &=& (xy)^{(p-2)k+1} ( (xy)^{k} x )^{-(p-3)}  (xy)^{(p-1)k+1}  x \\
&=&  x (xy)^{(p-1)k+1}  x  (xy)^{-(p-1)k-1}  (xy)^{k} x (xy)^{(p-1)k+1}  x.
\end{eqnarray*} 
Let $u=(xy)^{(p-1)k+1}$.  Since $x \alpha = \alpha$ and $y\alpha > \alpha$ we have $u \alpha > \alpha$. Hence 
\begin{eqnarray*}
&&\alpha = \lambda' \alpha = x u x u^{-1}  (xy)^k x u x \alpha \\
\Rightarrow && \alpha = u x u^{-1}  (xy)^k x u \alpha \\
\Rightarrow &&   (xy)^k x  u \alpha = u x^{-1} u^{-1} \alpha < ux^{-1}  \alpha = u \alpha \\
\Rightarrow && ( (xy)^k x )^{-1} u \alpha > u \alpha.
\end{eqnarray*}
This implies that $( (xy)^{k} x )^{-(p-3)}  (xy)^{(p-1)k+1}  \alpha = ( (xy)^{k} x)^{-(p-3)} u  \alpha > u \alpha > \alpha.$ Hence $$\lambda' \alpha =(xy)^{(p-2)k+1} ( (xy)^{k} x )^{-(p-3)}  (xy)^{(p-1)k+1}  x \alpha> \alpha.$$
This contradicts $\lambda' \alpha =\alpha.$
\end{proof}

Since $v\alpha \not=\alpha$ for all $\alpha \in \BR$ and $v\alpha$ is a continuous function of $\alpha$, without loss of generality, we may assume $v \alpha > \alpha$ for all $\alpha \in \BR$. Then $x \alpha = v^s \alpha > \alpha.$

\begin{proposition} \label{key0}
With $\frac{r}{s} \ge 2 g(K)-1$ we have $\alpha < y \alpha< x \alpha$ for all $\alpha \in \BR$.
\end{proposition}

\begin{proof}
With $\frac{r}{s} \ge 2 g(K)-1$ and $s>0$, we have $-r + (2g(K)-1)s \le  0$. Since $x = v^s$, $\lambda = v^{-r}$ and $v\alpha > \alpha$ for all $\alpha \in \BR$, we have 
$$x \alpha \ge v^{-r+(2 g(K)-1)s} x\alpha = x^{2g(K)}\lambda \alpha = x^{2g(K) -p(pk \pm 1) - l^2 m }\lambda' \alpha.$$

\underline{Case 1a}: $K=T^{p-1,m}_{p, pk - 1}.$ By Lemmas \ref{genus-1} and \ref{long} we have
\begin{eqnarray*}
2g(K) &=&  (p-1)(p-2)m +p(p-1)k-(2p-2), \\
\lambda' &=& (yx)^{k-1} y^{m+1} x (yx)^{k-1} (y^m (yx)^k)^{p-2}.
\end{eqnarray*}
By Lemma \ref{ineq-1} we have $\rho(y)\alpha > \alpha$, 
$(yx)^{k-1} y^n \alpha > x^n (yx)^{k-1}   \alpha$, $y^n (xy)^{k-1} \alpha  > (xy)^{k-1}  x^n  \alpha$
for all $\alpha \in \BR$ and $n \ge 1$. 
Hence 
\begin{eqnarray*}
x \alpha &\ge&  x^{-(p-1)m - p k - (p-2)} (yx)^{k-1} y^{m+1}  x(yx)^{k-1} ( y^m (yx)^k )^{p-2} \alpha \\
&=& x^{-(m+1)} (yx)^{k-1} y^{m+1}  x(yx)^{k-1} ( y^{m+1} (xy)^{k-1}x )^{p-2} x^{-(p-2)m - p k - (p-3)}  \alpha \\
&>& x^{-(m+1)} x^{m+1} (yx)^{k-1}  x(yx)^{k-1} ( (xy)^{k-1} x^{m+1} x )^{p-2} x^{-(p-2)m - p k - (p-3)}  \alpha \\
&>& y x^{k-1} x x^{k-1} ( x^{k-1} x^{m+2}  )^{p-2} x^{-(p-2)m - p k - (p-3)} \alpha \\
&=& y \alpha.
\end{eqnarray*}

\underline{Case 1b}: $K=T^{p-1,m}_{p, pk + 1}.$ By Lemmas \ref{genus+1} and \ref{long} we have
\begin{eqnarray*}
2g(K) &=& (p-1)(p-2)m + p(p-1)k, \\
\lambda' &=& ((xy)^{k} x y^m )^{p-1} (xy)^k x.
\end{eqnarray*}
By Lemma \ref{ineq+1} we have $\rho(y)\alpha > \alpha$, $(yx)^{k} y^n \alpha > x^n (yx)^{k}   \alpha$, $y^n (xy)^{k} \alpha  > (xy)^{k}  x^n  \alpha$ for all $\alpha \in \BR$ and $n \ge 1$. 
Hence 
\begin{eqnarray*} 
x \alpha &\ge&  x^{- (p-1)m- p k - p} ( x (yx)^{k}y^{m} )^{p-1} (xy)^k x\alpha \\
&=& x^{- (p-1)m- p k - (p-1)} ( x (yx)^{k}y^{m} )^{p-1} (xy)^k x x^{-1}\alpha \\
&>& x^{- (p-1)m- p k - (p-1)} ( x x^{m} (yx)^k )^{p-1} (xy)^{k} \alpha \\
&>& x^{- (p-1)m- p k - (p-1)} ( x^{m+1} x^k )^{p-1} x^{k} y \alpha \\
&=& y \alpha.
\end{eqnarray*}

\underline{Case 2a}: $K=T^{p-2,1}_{p, pk - 1}.$ By Lemmas \ref{genus-1} and \ref{long} we have
\begin{eqnarray*}
2g(K) &=&p^2-7p+8+ p(p-1)k, \\
\lambda' &=& x (yx)^{k-1} (y (yx)^{k-1}y )^{p-2} (yx)^{k-1}y.
\end{eqnarray*}
By Lemma \ref{ineq-1} we have $\rho(y)\alpha > \alpha$, $(yx)^{k-1} y^n \alpha > x^n (yx)^{k-1}   \alpha$, $y^n (xy)^{k-1} \alpha  > (xy)^{k-1}  x^n  \alpha$, $((yx)^{k-1} y )^{2} \alpha > x (yx)^{2k-1}   \alpha$
for all $\alpha \in \BR$ and $n \ge 1$. The latter implies that $y (yx)^{k-1} y  \alpha > (yx)^{-(k-1)} x (yx)^k (yx)^{k-1}   \alpha.$
Hence 
\begin{eqnarray*}
x \alpha &\ge&  x^{-(pk+2p-4)} x (yx)^{k-1} (y (yx)^{k-1}y )^{p-2} (yx)^{k-1}y \alpha \\
&>& x^{-(pk+2p-5)}  (yx)^{k-1} [(yx)^{-(k-1)} x(yx)^k (yx)^{k-1}]^{p-2} (yx)^{k-1}y \alpha \\
&>& x^{-(pk+2p-5)}   ( x(yx)^k )^{p-2} (yx)^{k-1} (yx)^{k-1}y \alpha \\
&=& x^{-(pk+2p-5)}   ( x (yx)^{k-1} y x )^{p-2} (yx)^{k-1} y (xy)^{k-1} \alpha \\
&>& x^{-(pk+2p-5)}   ( x x (yx)^{k-1} x )^{p-2} x (yx)^{k-1} (xy)^{k-1} \alpha \\
&>& x^{-(pk+2p-5)}   (x^2 x^{k-1}x )^{p-2} x x^{k-1} x^{k-1}y \alpha \\
&=& y \alpha.
\end{eqnarray*}

\underline{Case 2b}: $K=T^{p-2,1}_{p, pk + 1}.$ By Lemmas \ref{genus+1} and \ref{long} we have
\begin{eqnarray*}
2g(K) &=& p^2-5p+6 + p(p-1)k, \\
\lambda' &=& (xy)^{k} x (y (xy)^{k})^{p-2} (xy)^k x.
\end{eqnarray*}
By Lemma \ref{ineq+1} we have $\rho(y)\alpha > \alpha$, $(yx)^{k} y^n \alpha > x^n (yx)^{k}   \alpha$, $y^n (xy)^{k} \alpha  > (xy)^{k}  x^n  \alpha$, $(xy)^{2k+1}  \alpha >  x ( (xy)^k x)^{2} \alpha$
for all $\alpha \in \BR$ and $n \ge 1.$ The latter is equivalent to $y (xy)^{2k}  \alpha > ( (xy)^k x)^{2} \alpha$, which implies that  $y (xy)^{k}  \alpha > ( (xy)^k x)^{2} (xy)^{-k}\alpha.$ Hence 
\begin{eqnarray*}
x \alpha &\ge& x^{-(pk+ 2p-2)} (xy)^{k} x (y (xy)^{k})^{p-2} (xy)^k x \alpha \\
&=& x^{-1} x (yx)^{k}  (  y (xy)^{k} )^{p-2} (xy)^k x  x^{-(pk+ 2p-3)} \alpha \\
&>&  (yx)^{k} [ (xy)^k x (xy)^k x (xy)^{-k} ]^{p-2} (xy)^{k}  x^{-(pk+ 2p-4)} \alpha \\
&=&  (yx)^{k} (xy)^k ( x (xy)^k x )^{p-2}  x^{-(pk+ 2p-4)} \alpha \\
&>& y x^{k} x^{k} (x x^{k}x)^{p-2}  x^{-(pk+ 2p-4)} \alpha \\
&=& y \alpha.
\end{eqnarray*}

\underline{Case 3a}: $K=T^{2,1}_{p, pk - 1}.$ By Lemmas \ref{genus-1} and \ref{long} we have
\begin{eqnarray*}
2g(K) &=&  p(p-1)k -(2p-4), \\
\lambda' &=& x ( (yx)^{k-1}y )^{p-1}x ((yx)^{k-1} y )^{-(p-1)}  (yx)^k ( (yx)^{k-1}y )^{p-1} x.
\end{eqnarray*}
By Lemma \ref{ineq-1} we have $\rho(y)\alpha > \alpha$, $(yx)^{k-1} y^n \alpha > x^n (yx)^{k-1}   \alpha$, $y^n (xy)^{k-1} \alpha  > (xy)^{k-1}  x^n  \alpha$, $( (yx)^{k-1} y )^{p-2}\alpha >  x (yx)^{k(p-2)-1} \alpha$
for all $\alpha \in \BR$ and $n \ge 1.$ The latter implies that $( (yx)^{k-1} y )^{p-1} x\alpha >  x (yx)^{k(p-2)-1} (yx)^{k} \alpha = x (yx)^{k-1} (yx)^{k(p-2)} \alpha.$
Hence 
\begin{eqnarray*}
x \alpha &\ge& x^{-(pk+p)} x ( (yx)^{k-1}y )^{p-1}x ((yx)^{k-1} y )^{-(p-1)}  (yx)^k ( (yx)^{k-1}y )^{p-1} x \alpha \\
&=& x^{-2}   x  ( (yx)^{k-1}y )^{p-1}x ((yx)^{k-1} y )^{-(p-1)} (yx)^k ( (yx)^{k-1}y )^{p-1} x x^{-(pk+p-2)} \alpha \\
&>& x^{-1}   (yx)^k  ( (yx)^{k-1}y )^{p-1}  x x^{-(pk+p-2)} \alpha \\
&>&  x^{-1}  (yx)^k x (yx)^{k-1}  (yx)^{k(p-2)}  x^{-(pk+p-2)} \alpha \\
&>& x^{-1}  x y x^k x x^{k-1} ( x y x^k)^{p-2}  x^{-(pk+p-2)}\alpha \\
&>& y x^{2k} x^{(k+1)(p-2)}  x^{-(pk+p-2)}\alpha \\
&=& y \alpha.
\end{eqnarray*}
Here we use $(yx)^k \alpha = (yx)^{k-1} y x \alpha > x(yx)^{k-1} x \alpha> x y x^k \alpha$ for all $\alpha \in \BR$.

\underline{Case 3b}: $K=T^{2,1}_{p, pk + 1}.$ By Lemmas \ref{genus+1} and \ref{long} we have
\begin{eqnarray*}
2g(K) &=& 2+ p(p-1)k, \\
\lambda' &=&x (xy)^{(p-1)k+1}  x  (xy)^{-(p-1)k-1}  (xy)^{k} x (xy)^{(p-1)k+1}  x.
\end{eqnarray*}
By Lemma \ref{ineq+1} we have $\rho(y)\alpha > \alpha$, $(yx)^{k} y^n \alpha > x^n (yx)^{k}   \alpha$, $y^n (xy)^{k} \alpha  > (xy)^{k}  x^n  \alpha$, $(xy)^{k(p-2)+1} \alpha > x ((xy)^k x)^{p-2}  \alpha$
for all $\alpha \in \BR$ and $n \ge 1$. The latter implies that $(xy)^{k(p-1)+1} \alpha >  ((xy)^k x)^{p-1}  \alpha.$ Hence 
\begin{eqnarray*}
x \alpha &\ge&  x^{-(pk+p+2)} x (xy)^{(p-1)k+1}  x  (xy)^{-(p-1)k-1}  (xy)^{k} x (xy)^{(p-1)k+1}  x \alpha \\
&=&  x^{-(pk+p+1)} x (xy)^{(p-1)k+1}  x  (xy)^{k}  (xy)^{-(p-1)k-1} x (xy)^{(p-1)k+1} x x ^{-1} \alpha \\
&>& x^{-(pk+p+1)} x (xy)^{(p-1)k+1}  x (xy)^k  \alpha \\
&>&  x^{-(pk+p+1)} x ((xy)^k x)^{p-1}   x (xy)^k  \alpha \\
&>& x^{-(pk+p+1)} x (x^k x)^{p-1}   x x^k  y \alpha \\
&=& y \alpha.
\end{eqnarray*}
This completes the proof of Proposition \ref{key0}. 
\end{proof}

By Lemmas \ref{ineq-1} and \ref{ineq+1} we have $(xy)^j x \alpha < y(xy)^j \alpha$ for all $\alpha \in \BR$, where $j = k-1$ if $K = T^{l,m}_{p,pk-1}$ and $j=k$ if $K = T^{l,m}_{p,pk+1}.$ With $\frac{r}{s} \ge 2 g(K)-1$, by Proposition \ref{key0} we have $x \alpha > y \alpha$ for all $\alpha \in \BR$. Hence 
\begin{eqnarray*}
(xy)^j \alpha &=& (xy)^j x x^{-1} \alpha \\
                &<& y(xy)^j x^{-1} \alpha = (yx)^j y x^{-1} \alpha \\
                &<& (yx)^j x x^{-1} \alpha = x^{-1} (xy)^j x  \alpha \\
                &<& x^{-1} y(xy)^j \alpha \\
                &<& x^{-1} x (xy)^j \alpha = (xy)^j \alpha,
\end{eqnarray*}
a contradiction. This completes the proof of Theorem \ref{main}.

\section*{Acknowledgements} 
The author has been partially supported by a grant from the Simons Foundation (\#354595 to AT).

\end{document}